\documentclass[amssymb, 11pt]{amsart}
\usepackage{latexsym}

\newcommand{\F}{\mathbb{F}}
\newcommand{\Irr}{\mathrm{Irr}}

\newcommand{\qb}[2]{{\left [{#1 \atop #2} \right]}}
\newlength{\standardunitlength}
\newtheorem{prop}{Proposition}[section]

\newtheorem{lemma}[prop]{Lemma}
\newtheorem{cor}[prop]{Corollary}
\newtheorem{theorem}[prop]{Theorem}

\allowdisplaybreaks[1]

\begin{document}

\title [Generating functions for real character degree sums] {Generating functions for real character degree sums of finite general linear and unitary groups}

\author{Jason Fulman}
\address{Department of Mathematics\\
        University of Southern California\\
        Los Angeles, CA, 90089}
\email{fulman@usc.edu}

\author{C. Ryan Vinroot}
\address{Department of Mathematics\\College of William and Mary\\ Williamsburg, VA, 23187}
\email{vinroot@math.wm.edu}


\thanks{{\it Date}: May 31, 2013}

\begin{abstract} We compute generating functions for the sum of the real-valued character degrees of the finite general linear and unitary groups, through symmetric function computations.  For the finite general linear group, we get a new combinatorial proof that every real-valued character has Frobenius-Schur indicator 1, and we obtain some $q$-series identities.  For the finite unitary group, we expand the generating function in terms of values of Hall-Littlewood functions, and we obtain combinatorial expressions for the character degree sums of real-valued characters with Frobenius-Schur indicator $1$ or $-1$.
\end{abstract}

\maketitle

\section{Introduction}

Suppose that $G$ is a finite group, $\Irr(G)$ is the set of irreducible complex characters of $G$, and $\chi \in \Irr(G)$ affords the representation $(\pi, V)$.  Recall that the \emph{Frobenius-Schur indicator} of $\chi$ (or of $\pi$), denoted $\varepsilon(\chi)$, takes only the values $1$, $-1$, or $0$, where $\varepsilon(\chi) = \pm 1$ if and only if $\chi$ is real-valued, and $\varepsilon(\chi) = 1$ if and only if $(\pi, V)$ is a real representation \cite[Theorem 4.5]{Isa}.  From the formula $\varepsilon(\chi) = (1/|G|) \sum_{g \in G} \chi(g^2)$ \cite[Lemma 4.4]{Isa}, it follows that we have \cite[Corollary 4.6]{Isa}
\begin{equation} \label{invol}
\sum_{\chi \in \Irr(G)} \varepsilon(\chi) \chi(1) = \sum_{\chi \in \Irr(G) \atop{ \varepsilon(\chi) = 1}} \chi(1)- \sum_{\chi \in \Irr(G) \atop {\varepsilon(\chi) = -1}} \chi(1)= \# \{ h \in G \, \mid \, h^2 = 1 \}.
\end{equation}
In particular, from (\ref{invol}) it follows that the statement that $\varepsilon(\chi) = 1$ for every real-valued $\chi \in \Irr(G)$ is equivalent to the statement that the sum of the degrees of all real-valued $\chi \in \Irr(G)$ is equal to the right side of (\ref{invol}).

The main topic of this paper is to study sums of degrees of real-valued characters from a combinatorial point of view of the general linear and unitary groups defined over a finite field $\F_q$ with $q$ elements, which we denote by $GL(n,q)$ and $U(n,q)$, respectively.

Before addressing the main question of the paper, we begin in Section \ref{Weyl} by examining the classical Weyl groups.  In particular, we consider the Weyl groups of type $A$ (symmetric groups), type $B$/$C$ (hyperoctahedral groups), and type $D$.  It is well known that every complex irreducible representation of every finite Coxeter group is defined over the real numbers, and moreover, it was proved using a unified method by Springer that all complex representations of Weyl groups are defined over $\mathbb{Q}$ \cite{Sp}.  We consider the classical Weyl groups, however, as they serve as natural examples to demonstrate the method of generating functions and symmetric function identities to calculate the (real) character degree sum of a group.

In Section \ref{GLnSum}, we concentrate on the real character degree sum of $GL(n, q)$.  It is in fact known that $\varepsilon(\chi) = 1$ for every real-valued $\chi \in \Irr(GL(n,q))$.  This was first proved for $q$ odd by Ohmori \cite{Ohm77}, and it follows for all $q$ by a result of Zelevinsky \cite[Proposition 12.6]{Zel}.  So the real character degree sum is known to be the number of elements of $GL(n,q)$ which square to 1.  We obtain this result independently by calculating the sum of the degrees of the real-valued characters through symmetric function computations, and applying $q$-series identities.  We also obtain some $q$-series identities in the process.  In Theorem \ref{genfnGL}, we give a generating function for the real character degree sum of $GL(n,q)$ from symmetric function calculations.  In Theorem \ref{even}, we recover Zelevinsky's result on the Frobenius-Schur indicators of characters of $GL(n,q)$ in the case that $q$ is even, and we record the corresponding $q$-series identity in Corollary \ref{iden}.  We recover Zelevinsky's result for the case that $q$ is odd in Theorem \ref{odd}, where the calculation is a bit more involved than in the case that $q$ is even.  The resulting $q$-series identity, in Corollary \ref{cort}, seems to be an interesting result in its own right.

In Section \ref{unitary}, we turn to the problem of calculating the real character degree sum for the finite unitary group $U(n,q)$.  The main motivation here is that given a real-valued $\chi \in \Irr(U(n,q))$, it is unknown in general whether $\varepsilon(\chi) = 1$ or $\varepsilon(\chi) = -1$, although the values of $\varepsilon(\chi)$ are known for certain subsets of characters of $U(n,q)$, such as the unipotent characters \cite{Ohm} and the regular and semisimple characters \cite{SV}.  Unlike the $GL(n,q)$ case, it is known that there are $\chi \in \Irr(U(n,q))$ such that $\varepsilon(\chi) = -1$.  Using symmetric function techniques similar to the $GL(n,q)$ case, we compute a generating function for the real character degree sum of $U(n,q)$ in Theorem \ref{degreesU}.  The point here is that from Equation (\ref{invol}), by counting the number of elements in $U(n,q)$ which square to 1, we have the difference of the real character degree sums of those $\chi$ such that $\varepsilon(\chi) = 1$, minus those $\chi$ such that $\varepsilon(\chi)=-1$.  Using the generating function in Theorem \ref{degreesU} for the sum (rather than the difference) of these character degree sums, along with applying the $q$-series identities obtained in the $GL(n,q)$ case, we obtain a generating function for the character degree sums of $\chi \in \Irr(U(n,q))$ satisfying $\varepsilon(\chi) = 1$, and another for those satisfying $\varepsilon(\chi) = -1$, in Corollaries \ref{GenFnEpEven} and \ref{GenFnEpOdd}.

We are then able to expand the generating function obtained in Theorem \ref{degreesU}, using results of Warnaar \cite{War}, stated in Theorem \ref{WarId} and Corollary \ref{WarCor}.  For $q$ even, we give the resulting expression for the sum of the real character degrees of $U(n,q)$ in Theorem \ref{UnSumEven}, and for $q$ odd, in Theorem \ref{UnSumOdd}.  These expressions contain, among other things, special values of Hall-Littlewood functions of the form $P_{\lambda}(1, t, t^2, \ldots; -t)$, where $t = -q^{-1}$.  Since these values do not seem to be well understood, the expressions we obtain are somewhat complicated in terms of calculation, but we compute several examples to verify the expressions for small $n$.  The fact that these values of Hall-Littlewood functions show up in representation theory gives motivation to better understand them, and the fact that the expressions we obtain are complicated may reflect the fact that it has been a difficult problem to understand the Frobenius-Schur indicators of the characters of $U(n,q)$.  We hope that a better understanding of the combinatorial expressions we obtain in Theorems \ref{UnSumEven} and \ref{UnSumOdd} will reveal interesting character-theoretic information for the finite unitary groups.

\section{Examples:  Classical Weyl groups} \label{Weyl}

As was mentioned above, it is well understood that every complex irreducible character $\chi$ of a finite Coxeter group satisfies $\varepsilon(\chi) = 1$.  To motivate the methods we will use for the groups $GL(n,q)$ and $U(n,q)$, we consider this fact for the classical Weyl groups, from the perspective of generating functions and Schur function identities.

\subsection{Symmetric groups}  It is well known (see \cite{CHM}, for example) that the number of elements which square to the identity in the symmetric group $S_n$ (the Weyl group of type $A_{n-1}$) is exactly $n!$ times the coefficient of $u^n$ in $e^{u + u^2/2}$.

On the other hand, the irreducible complex representations of $S_n$ are parameterized by partitions $\lambda$ of $n$, and if $d_{\lambda}$ is the degree of the irreducible character $\chi_{\lambda}$ labeled by the partition $\lambda$ of $n$, then $d(\lambda)$ is given by the hook-length formula, $d_{\lambda} = n!/\prod_{b \in \lambda} h(b)$, where $h(b)$ is the hook length of a box $b$ in the diagram for $\lambda$.  By \cite[I.3, Example 5]{Mac}, we have
$$ d_{\lambda} = \frac{n!}{\prod_{b \in \lambda} h(b)} = n! \cdot \lim_{m \rightarrow \infty} s_{\lambda}(1/m, \ldots, 1/m),$$
where $s_{\lambda} (1/m, \ldots, 1/m)$ is the Schur function in $m$ variables, evaluated at $1/m$ for each variable.  The Schur symmetric functions are studied extensively in \cite[Chapter I]{Mac}.  Using the identity \cite[I.5, Example 4]{Mac}
\begin{equation} \label{Schur1}
\sum_{\lambda} s_{\lambda}(x) = \prod_i (1 - x_i)^{-1} \prod_{i < j} (1 - x_i x_j)^{-1},
\end{equation}
we may compute that the sum $\sum_{|\lambda|=n} d_{\lambda}$ is $n!$ times the coefficient of $u^n$ in
\begin{align*}
\lim_{m \rightarrow \infty} \sum_{\lambda} s_{\lambda}(u/m, \ldots, u/m)  & = \lim_{m \rightarrow \infty} (1 - u/m)^{-m} (1-u^2/m^2)^{-\binom{m}{2}} \\
 & = e^{u + u^2/2}.
\end{align*}
It follows that the sum of the degrees of the irreducible characters of $S_n$ is equal to the number of elements in $S_n$ which square to the identity.  From the Frobenius-Schur theory, this is equivalent to $\varepsilon(\chi) = 1$ for every complex irreducible character $\chi$ of $S_n$.

\subsection{Hyperoctahedral groups}  The hyperoctahedral group $(\mathbb{Z}/2\mathbb{Z}) \wr S_n$ is isomorphic to the Weyl group of type $B_n$ (or type $C_n$), which we denote by $W(B_n)$.  It follows from \cite[Theorem 2]{Ch} that the number of elements in this group with square the identity is $n!$ times the coefficient of $u^n$ in $e^{2u + u^2}$.

From \cite[I.9]{Mac}, the complex irreducible characters of $W(B_n)$ can be usefully parameterized by ordered pairs $(\lambda, \tau)$ of partitions such that $|\lambda| + |\tau| = n$.  If $d_{\lambda, \tau}$ is the degree of the irreducible character labeled by the pair $(\lambda, \tau)$, then by \cite[I.9, Equation (9.6)]{Mac} we have
$$ d_{\lambda, \tau} = \frac{n!}{\prod_{b \in \lambda} h(b) \prod_{b \in \tau} h(b)}.$$
Using the same Schur polynomial identity (\ref{Schur1}) as we did for the symmetric groups, we may compute the sum $\sum_{|\lambda| + |\tau| = n} d_{\lambda, \tau}$ to be $n!$ times the coefficient of $u^n$ in
\begin{align*}
\sum_{\lambda} \frac{u^{|\lambda|}}{\prod_{b \in \lambda} h(b)}  &\sum_{\tau} \frac{u^{|\tau|}}{\prod_{b \in \tau} h(b)}  \\
& = \lim_{m \rightarrow \infty} \sum_{\lambda} s_{\lambda}(u/m, \ldots, u/m) \sum_{\tau} s_{\tau} (u/m, \ldots, u/m) \\
& = \lim_{m \rightarrow \infty} (1 - u/m)^{-2m} (1 - u^2/m^2)^{-2 \binom{m}{2}} = e^{2u + u^2}.
\end{align*}
Again, from the equality of the two generating functions, the sum of the degrees of the irreducible characters of $W(B_n)$ is equal to the number of elements in the group which square to the identity, which is equivalent to the statement that $\varepsilon(\chi) = 1$ for every complex irreducible character $\chi$ of the group.

\subsection{Weyl groups of type $D$}  As a final example, we consider the Weyl group of type $D_n$, denoted by $W(D_n)$, which is the index 2 subgroup of $W(B_n) \cong (\mathbb{Z}/2\mathbb{Z}) \wr S_n$, consisting of elements with an even number of negative cycles.

It follows from \cite[Theorem 5]{Ch} that the number of elements in $W(D_n)$ whose square is the identity is $n!/2$ times the coefficient of $u^n$ in $e^{u^2}(e^{2u} + 1)$.

The complex irreducible characters of $W(D_n)$ may be described as follows \cite[Proposition 11.4.4]{Ca}.  For any partition $\lambda$ of $n/2$, when $n$ is even, the irreducible representation of $W(B_n)$ labeled by $(\lambda, \lambda)$ restricts to a direct sum of two non-isomorphic irreducible representations of the same dimensions of $W(D_n)$.  For any pair of partitions $(\mu, \tau)$ such that $|\mu| + |\tau| = n$ and $\mu \neq \tau$, the irreducible representations of $W(B_n)$ labeled by $(\mu, \tau)$ and $(\tau, \mu)$ both restrict to isomorphic irreducible representations of $W(D_n)$.  These account for all distinct irreducible representations of the Weyl group of type $D_n$.

Using this parametrization of the irreducible characters of $W(D_n)$, and the formula for the character degrees for the irreducible characters of $W(B_n)$, the sum of the degrees of the irreducible characters of $W(D_n)$ is $n!/2$ times the coefficient of $u^n$ in
\begin{align*}
\lim_{m \rightarrow \infty} &\sum_{\lambda} s_{\lambda}(u/m, \ldots, u/m) s_{\lambda}(u/m, \ldots, u/m) \\
&+ \lim_{m \rightarrow \infty} \sum_{\mu} s_{\mu} (u/m, \ldots, u/m) \sum_{\tau} s_{\tau}(u/m, \ldots, u/m).
\end{align*}

We apply the identity (\ref{Schur1}), along with the Schur function identity \cite[I.4, Equation (4.3)]{Mac}
\begin{equation} \label{Schur2}
\sum_{\lambda} s_{\lambda}(x) s_{\lambda}(y) = \prod_{i,j} (1 - x_i y_j)^{-1},
\end{equation}
to rewrite the previous expression as
$$\lim_{m \rightarrow \infty} (1 - u^2/m^2)^{-m^2} + \lim_{m \rightarrow \infty} (1-u/m)^{-2m} (1 - u^2/m^2)^{-2 \binom{m}{2}} = e^{u^2} (e^{2u} + 1).$$
By matching coefficients of generating functions, we again have that the sum of the irreducible character degrees is equal to the number of elements which square to the identity, in the Weyl group of type $D_n$, so $\varepsilon(\chi) = 1$ for every irreducible character $\chi$ of this group.

\section{Real character degree sums for $GL(n,q)$} \label{GLnSum}

  We require some background on polynomials. For a monic polynomial $\phi(t) \in \mathbb{F}_q[t]$ of degree $n$ with non-zero constant term, we define the $*$-conjugate $\phi^*(t)$ by
\[ \phi^*(t) := \phi(0)^{-1} t^n \phi(t^{-1}).\] Thus if
\[ \phi(t) = t^n + a_{n-1}t^{n-1} + \cdots + a_1 t + a_0 \] then
\[ \phi^*(t) = t^n + a_1a_0^{-1} t^{n-1} + \cdots + a_{n-1}a_0^{-1} t + a_0^{-1}. \] We say that $\phi$ is self-conjugate (or $*$-self conjugate) if $\phi(0) \neq 0$ and $\phi^*=\phi$.

We let $N^*(d,q)$ denote the number of monic irreducible self-conjugate polynomials $\phi(t)$ of degree $d$ over $\mathbb{F}_{q}$, and let $M^*(d,q)$ denote the number of (unordered) conjugate pairs $\{ \phi,\phi^* \}$ of monic irreducible polynomials of degree $d$ over $\mathbb{F}_{q}$ that are not self conjugate. The following lemma \cite[Lemma 1.3.17 (a) and (d)]{FNP} will be helpful.

\begin{lemma} \label{prodlem} Let $e=1$ if the characteristic is even, and $e=2$ if the characteristic is odd.
\begin{enumerate}
\item $\prod_{d \geq 1} (1-w^d)^{-N^*(2d,q)} (1-w^d)^{-M^*(d,q)} = \frac{(1-w)^e}{1-qw}$
\item $\prod_{d \geq 1} (1+w^d)^{-N^*(2d,q)} (1-w^d)^{-M^*(d,q)} = 1-w$
\end{enumerate}
\end{lemma}

We now obtain a generating function for the sum of the degrees of the real-valued irreducible complex characters of $GL(n,q)$.

\begin{theorem} \label{genfnGL} Let $e=1$ if $q$ is even, and $e=2$ if $q$ is odd.  The sum of the degrees of the real-valued characters of $GL(n,q)$ is $(q^n - 1) \cdots (q-1)$ times the coefficient of $u^n$ in
\[ \prod_{i \geq 1} \frac{(1 + u/q^i)^e}{1 - u^2/q^i}. \]
\end{theorem}
\begin{proof} From \cite[Chapter IV]{Mac}, the irreducible characters of $GL(n,q)$ correspond to choosing a partition $\lambda(\phi)$ for each monic, irreducible polynomial $\phi(t) \neq t$, in such a way that $\sum_{\phi} d(\phi)|\lambda(\phi)| = n$, where $d(\phi)$ is the degree of $\phi$. From \cite[pg. 286]{Mac}, the degree of such a character is:

  \[ (q^n-1) \cdots (q^2-1)(q-1) \prod_{\phi} \frac{q^{d(\phi) n(\lambda(\phi)')}}{\prod_{b \in \lambda(\phi)} (q^{d(\phi) h(b)}-1)} .\]
Here, $h(b)$ is a hook-length as in Section \ref{Weyl}, $\lambda'$ is the conjugate partition of $\lambda$, and $n(\lambda)$ is the statistic defined as $n(\lambda) = \sum_i (i-1) \lambda_i$, where $\lambda_i$ is the $i$th part of $\lambda$. By \cite[pg. 11]{Mac}, one has that
$\sum_{b \in \lambda} h(b) = n(\lambda) + n(\lambda') + |\lambda|$, so it follows from \cite[pg. 45]{Mac} that the character
  degree can be rewritten as:
  \[ \frac{(q^n-1) \cdots (q^2-1)(q-1)}{q^n} \prod_{\phi} s_{\lambda(\phi)}(1,1/q^{d(\phi)},1/q^{2d(\phi)},\cdots) ,\] where $s_{\lambda}$ is a Schur function, as in Section \ref{Weyl}.

It is known, by \cite[1.1.1]{BT} and \cite[Lemma 3.1]{GV}, that a character of $GL(n,q)$ is real precisely when $\lambda(\phi)=\lambda(\phi^*)$ for all monic irreducible polynomials $\phi$. Thus the sum of the degrees of the real characters of $GL(n,q)$ is equal to $(q^n-1) \cdots (q-1) / q^n$ multiplied by the coefficient of $u^n$
  in

 \begin{eqnarray*}
 & & \prod_{d \geq 1} \left[ \sum_{\lambda} u^{d |\lambda|} s_{\lambda}(1,1/q^d,1/q^{2d},\cdots) \right] ^{N^*(d,q)} \\
 & & \cdot \prod_{d \geq 1} \left[ \sum_{\lambda} u^{2 d |\lambda|} s_{\lambda}(1,1/q^d,1/q^{2d},\cdots)^2 \right] ^{M^*(d,q)}.
 \end{eqnarray*}

This is $(q^n-1) \cdots (q-1)$ multiplied by the coefficient of $u^n$ in

   \[ \prod_{d \geq 1} \left[ \sum_{\lambda} s_{\lambda}(u^d/q^d,u^d/q^{2d},\cdots) \right] ^{N^*(d,q)}
  \left[ \sum_{\lambda} s_{\lambda}(u^d/q^d,u^d/q^{2d},\cdots)^2 \right] ^{M^*(d,q)} \]

By the Schur function identities \eqref{Schur1} and \eqref{Schur2}, it follows that the sum
of the degrees of the real characters of $GL(n,q)$ is $(q^n-1) \cdots (q-1)$ multiplied by the coefficient
of $u^n$ in

\begin{eqnarray*}
A(u) & := & \prod_d \left[ \prod_i (1-u^d/q^{id})^{-1} \prod_{i<j} (1-u^{2d}/q^{(i+j)d})^{-1} \right]^{N^*(d,q)} \\
& & \cdot \prod_d \left[ \prod_{i,j} (1-u^{2d}/q^{(i+j)d})^{-1} \right] ^{M^*(d,q)}.
\end{eqnarray*}

Then $A(u) = B(u) C(u)$, where

\[ B(u) := \prod_i \prod_d (1-u^d/q^{id})^{-N^*(d,q)} (1-u^{2d}/q^{2id})^{-M^*(d,q)}  \quad \text{ and } \]
\[ C(u) := \prod_{i<j} \prod_d (1-u^{2d}/q^{(i+j)d})^{-N^*(d,q)} (1-u^{2d}/q^{(i+j)d})^{-2 M^*(d,q)}. \]

From \cite[Lemma 1.3.16]{FNP}, $N^*(1,q)=e$ and $N^*(d,q)=0$ for $d>1$ odd. Using this
together with part (1) of Lemma \ref{prodlem} (with $w=u^2/q^{2i}$) gives that
\begin{eqnarray*}
B(u) & = & \prod_i (1-u/q^i)^{-e} \prod_{d} (1-u^{2d}/q^{2id})^{-N^*(2d,q)} (1-u^{2d}/q^{2id})^{-M^*(d,q)} \\
& = & \prod_i (1-u/q^i)^{-e} \frac{(1-u^2/q^{2i})^e}{(1-u^2/q^{2i-1})} \\
& = & \prod_i \frac{(1+u/q^i)^e}{(1-u^2/q^{2i-1})}.
\end{eqnarray*}

Similarly, but using both parts of Lemma \ref{prodlem} (with $w=u^2/q^{i+j}$), one has that $C(u)$ is equal to
\begin{eqnarray*}
&  & \prod_{i<j} \left( 1-\frac{u^2}{q^{i+j}} \right)^{-e} \prod_d \left( 1- \frac{u^{4d}}{q^{2d(i+j)}} \right)^{-N^*(2d,q)}
\left( 1- \frac{u^{2d}}{q^{d(i+j)}} \right)^{-2M^*(d,q)} \\
& = & \prod_{i<j} \left( 1-\frac{u^2}{q^{i+j}} \right)^{-e}   \prod_d \left( 1- \frac{u^{2d}}{q^{d(i+j)}} \right)^{-N^*(2d,q)}
\left( 1- \frac{u^{2d}}{q^{d(i+j)}} \right)^{- M^*(d,q)} \\
& & \cdot \prod_d (1+u^{2d}/q^{d(i+j)})^{-N^*(2d,q)}
(1-u^{2d}/q^{d(i+j)})^{- M^*(d,q)} \\
& = & \prod_{i<j} (1-u^2/q^{i+j})^{-e} \frac{(1-u^2/q^{i+j})^e}{(1-u^2/q^{i+j-1})} (1-u^2/q^{i+j}) \\
& = & \prod_{i<j} \frac{(1-u^2/q^{i+j})}{(1-u^2/q^{i+j-1})} \\
& = & \prod_{i} (1-u^2/q^{2i})^{-1}.
\end{eqnarray*}

Combining these expressions for $B(u)$ and $C(u)$ gives that
\[ A(u) = \prod_i (1+u/q^i)^e \prod_i (1-u^2/q^i)^{-1} ,\]
giving the claim.  \end{proof}

We now consider the implications of Theorem \ref{genfnGL} in the cases that $q$ is even or odd separately.  In both cases, we can recover the result that $\varepsilon(\chi) = 1$ for every real-valued irreducible character $\chi$ of $GL(n,q)$, through a combinatorial proof.

Throughout the paper from here, we denote
$$\gamma_j = |GL(j,q)| = q^{{j \choose 2}} (q^j-1) \cdots (q-1) = q^{j^2}(1-1/q) \cdots (1 - 1/q^j),$$ and we set $\gamma_0 =1$.
The following lemma, (Corollary 2.2 of \cite{A}) will be helpful.

\begin{lemma} \label{eul} For $|t|<1,|q|<1$,
\begin{enumerate}
\item \[ 1 + \sum_{n=1}^{\infty} \frac{t^n}{(1-q)(1-q^2) \cdots (1-q^n)} = \prod_{n=0}^{\infty} (1-tq^n)^{-1} \]
\item \[ 1 + \sum_{n=1}^{\infty} \frac{t^n q^{n(n-1)/2}}{(1-q)(1-q^2) \cdots (1-q^n)} = \prod_{n=0}^{\infty} (1+tq^n) \]
\end{enumerate}
\end{lemma}

Next we prove one of the main results of this section.

\begin{theorem} \label{even} Let $q$ be even.  Then sum of the degrees of the real-valued irreducible characters of $GL(n,q)$ is equal to
the number of elements of $GL(n,q)$ whose square is the identity.  So, for any real-valued irreducible character $\chi$ of $GL(n,q)$, $\varepsilon(\chi) = 1$.  In particular, we have
that the sum of the degrees of the real-valued irreducible characters of $GL(n,q)$ is equal to $(q^n-1) \cdots (q-1)$ times the coefficient of $u^n$ in
\[ \frac{ \prod_{i \geq 1} (1+u/q^i) }{ \prod_{i \geq 1} (1-u^2/q^i) },\] and also to
\[ \sum_{r =0}^{\lfloor n/2 \rfloor} \frac{\gamma_n}{q^{r(2n-3r)} \gamma_r \gamma_{n-2r}}. \]
\end{theorem}

\begin{proof} By Theorem \ref{genfnGL}, in the case that $q$ is even, we have that the sum of the degrees of the real-valued characters of $GL(n,q)$ is $(q^n - 1) \cdots (q-1)$ times the coefficient in $u^n$ in
\[ \frac{\prod_{i \geq 1} (1+u/q^i) }{ \prod_{i \geq 1} (1-u^2/q^i) }.\]
From Lemma \ref{eul},

\[ \prod_i (1+u/q^i) = \sum_{l \geq 0} \frac{u^l}{q^{{l+1 \choose 2}} (1-1/q)(1-1/q^2) \cdots (1-1/q^l)},  \text{ and }\]
\[ \prod_i (1-u^2/q^i)^{-1} = \sum_{r \geq 0} \frac{u^{2r}}{q^r (1-1/q) \cdots (1-1/q^r)}. \]

Thus the sum of the degrees of the real irreducible characters of $GL(n,q)$ is equal to
$(q^n-1) \cdots (q-1)$ multiplied by
\begin{equation} \label{one}  \sum_{r=0}^{\lfloor n/2 \rfloor} \frac{1}{q^r (1-1/q) \cdots (1-1/q^r)} \frac{1}{q^{{n-2r+1 \choose 2}}(1-1/q) \cdots (1-1/q^{n-2r})}. \end{equation}

From \cite[Section 1.11]{Mo}, in even characteristic the number of elements of $GL(n,q)$ whose square is the identity is
\begin{equation} \label{two} \sum_{r=0}^{\lfloor n/2 \rfloor} \frac{\gamma_n}{q^{r(2n-3r)} \gamma_r \gamma_{n-2r}}.\end{equation}

It is straightforward to see that $(q^n-1) \cdots (q-1)$ multiplied by the rth term in \eqref{one} is equal to
the rth term in \eqref{two}.  Thus, the sum of the degrees of the real-valued characters of $GL(n,q)$ is equal to the number of elements in $GL(n,q)$ whose square is the identity.  By the Frobenius-Schur theory, this is equivalent to the statement that $\varepsilon(\chi) = 1$ for any real-valued irreducible character $\chi$ of $GL(n,q)$.  This completes the proof.
\end{proof}

We record the following identity, which is directly implied by the proof of Theorem \ref{even} above.  We note that it will be used later in studying the real character degree sum of the finite unitary group in the next section.

\begin{cor} \label{iden} For any $q$, we have the formal identity
\[ \frac{ \prod_{i \geq 1}(1+u/q^i)}{\prod_{i \geq 1} (1-u^2/q^i)} = \sum_{n \geq 0} u^n q^{n \choose 2}
\sum_{r =0}^{\lfloor n/2 \rfloor} \frac{1}{q^{r(2n-3r)} \gamma_r \gamma_{n-2r}}. \]
\end{cor}

\noindent {\it Remark:} For $q$ even, let $i_{GL}(n,q)$ denote the number of involutions in $GL(n,q)$. The following are the values of $i_{GL}(n,q)$, for $1 \leq n \leq 7$.  The data suggests that it might be the case that when written as a polynomial in $q$, all coefficients of powers of $q$ are $0,-1$ or $1$:

\[ i_{GL}(1,q) = 1, \quad i_{GL}(2,q) = q^2, \quad i_{GL}(3,q) = q (-1 + q^2 + q^3), \]

\[ i_{GL}(4,q) = q^2 (-1 + q^4 + q^6), \quad i_{GL}(5,q) = q^6 (-1 - q + q^4 + q^5 + q^6), \]

\[ i_{GL}(6,q) = q^5 (1 - q^3 - q^4 - q^5 - q^6 + q^9 + q^{10} + q^{11} + q^{13}), \]

\[ i_{GL}(7,q) = q^7 (1 - q^6 - q^7 - q^8 - q^9 - q^{10} + q^{13} + q^{14} + q^{15} + q^{16} +
   q^{17} ). \]

Next we consider the case of odd characteristic.  In this case, we must work a bit harder than in the $q$ even case, but we still recover the result that $\varepsilon(\chi) = 1$ for every real-valued irreducible character $\chi$ of $GL(n,q)$.  We need the following notation.  For any $q$, we let $\qb{n}{i}_q$ denote the $q$-binomial coefficient, so for any integers $n \geq i \geq 0$,
$$\qb{n}{i}_q = \frac{(q^n - 1) \cdots (q-1)}{(q^i - 1) \cdots (q-1)(q^{n-i} - 1) \cdots (q-1)}.$$
Recall the $q$-binomial theorem \cite[Equation (3.3.6)]{A}, which we apply in the proof of the next result:
\begin{equation} \label{qbin}
(1+xq)(1+xq^2) \cdots (1+xq^n) = \sum_{i=0}^n \qb{n}{i}_q q^{i(i+1)/2} x^i.
\end{equation}
We now recover Zelevinsky's result for $GL(n,q)$ for the case that $q$ is odd.

\begin{theorem} \label{odd} Let $q$ be odd.  Then sum of the degrees of the real-valued irreducible characters of $GL(n,q)$ is equal to
the number of elements of $GL(n,q)$ whose square is the identity.  So, for any real-valued irreducible character $\chi$ of $GL(n,q)$, $\varepsilon(\chi) = 1$.  In particular, we have
that the sum of the degrees of the real-valued irreducible characters of $GL(n,q)$ is equal to $(q^n-1) \cdots (q-1)$ times the coefficient of $u^n$ in
\[ \frac{ \prod_{i \geq 1} (1+u/q^i)^2 }{ \prod_{i \geq 1} (1-u^2/q^i) },\] and also to
\[ \sum_{r =0}^{n} \frac{\gamma_n}{\gamma_r \gamma_{n-r}}. \]
\end{theorem}
\begin{proof}  We first claim that for any $q$, we have
\begin{equation} \label{OddId}
\sum_{n \geq 0} u^n \sum_{t=0}^n \sum_{s=0}^{n-t} \frac{q^{{n-t \choose 2}}} {\gamma_s \gamma_{n-t-s}} \frac{(-1)^t}{q^t(1-1/q) \cdots (1-1/q^t)} = \prod_{i \geq 1} \frac{(1+u/q^i)}{(1-u^2/q^i)}.
\end{equation}
The coefficient of $u^n$ on the left hand side of \eqref{OddId} is equal to
\begin{eqnarray*}
& & \sum_{t=0}^n q^{{n-t \choose 2}} \sum_{s=0}^{n-t} \frac{1}{\gamma_s \gamma_{n-t-s}} \frac{(-1)^t}{q^t (1-1/q) \cdots (1-1/q^t)} \\
& = & \sum_{s=0}^n \frac{1}{\gamma_s} \sum_{t=0}^{n-s} \frac{q^{{n-t \choose 2}}}{\gamma_{n-t-s}} \frac{(-1)^t}{q^t(1-1/q) \cdots (1-1/q^t)} \\
& = & \sum_{s=0}^n \frac{1}{\gamma_s (q^{n-s}-1) \cdots (q-1)} \sum_{t=0}^{n-s} \qb{n-s}{t}_q \frac{(-1)^t q^{{n-t \choose 2}} q^{{t+1 \choose 2}} }{q^{{n-t-s \choose 2}} q^t} \\
& = & \sum_{s=0}^n \frac{q^{ns-(s^2+s)/2}}{\gamma_s (q^{n-s}-1) \cdots (q-1)} \sum_{t=0}^{n-s} \qb{n-s}{t}_q (-1)^t q^{{t+1 \choose 2}} q^{-(s+1)t}.
\end{eqnarray*}
Now let $r = n-s$, and apply the $q$-binomial theorem (\ref{qbin}) with $x = -1/q^{n-r+1}$ to obtain
\begin{eqnarray*}
& & \sum_{r=0}^n \frac{q^{n(n-r) - ((n-r)^2+(n-r))/2}} {\gamma_{n-r} (q^r-1) \cdots (q-1)}
\sum_{t=0}^r \qb{r}{t}_q (-1)^t q^{{t+1 \choose 2}} q^{-(n-r+1)t} \\
& = & \sum_{r=0}^{\lfloor n/2 \rfloor} \frac{q^{n(n-r) - ((n-r)^2+(n-r))/2} (1-1/q^{n-r}) \cdots (1-1/q^{n-2r+1})}
{q^{(n-r)^2} (1-1/q^{n-r}) \cdots (1-1/q) (q^r-1) \cdots (q-1)} \\
& = &  \sum_{r=0}^{\lfloor n/2 \rfloor} \frac{q^{n(n-r) - ((n-r)^2+(n-r))/2} q^{(n-2r)^2} q^{{r \choose 2}}}{q^{(n-r)^2} \gamma_{n-2r} \gamma_r} \\
& = & q^{{n \choose 2}} \sum_{r=0}^{\lfloor n/2 \rfloor} \frac{1}{q^{r(2n-3r)} \gamma_r \gamma_{n-2r}}.
\end{eqnarray*}
By Corollary \ref{iden}, this is exactly the coefficient of $u^n$ in
\[ \prod_{i \geq 1} \frac{(1+u/q^i)}{(1-u^2/q^i)},\]
giving \eqref{OddId}.

Now, it follows from Lemma \ref{eul} that
$$\prod_{i \geq 1} (1 + u/q^i)^{-1} = \sum_{t \geq 0} \frac{(-1)^t u^t}{q^t (1 - 1/q) \cdots (1-1/q^t)},$$
and so the left side of \eqref{OddId} is
\[ \left[ \sum_{m \geq 0} q^{{m \choose 2}} \sum_{s=0}^m \frac{u^m}{\gamma_s \gamma_{m-s}} \right] \left[ \prod_{i \geq 1} (1+u/q^i)^{-1} \right],\]
from which it follows that
\begin{equation} \label{OddId2}
\prod_{i \geq 1} \frac{(1+u/q^i)^2}{(1-u^2/q^i)} = \sum_{n \geq 0} u^n q^{\binom{n}{2}} \sum_{r=0}^n \frac{1}{\gamma_r \gamma_{n-r}}.
\end{equation}
It is known (see \cite[Section 1.11]{Mo}, for example) that in odd characteristic, the number of elements of $GL(n,q)$ whose square is the identity is
$$\sum_{r=0}^n \frac{\gamma_n}{\gamma_r \gamma_{n-r}} = (q^n - 1) \cdots (q-1) q^{{n \choose 2}} \sum_{r=0}^n \frac{1}{\gamma_r \gamma_{n-r}}.$$
By Theorem \ref{genfnGL}, the sum of the degrees of the real-valued characters of $GL(n,q)$, when $q$ is odd, is $(q^n - 1) \cdots (q-1)$ times the coefficient of $u^n$ in the left side of \eqref{OddId2}.  This, together with \eqref{OddId2}, gives that this sum of character degrees is the number of elements in $GL(n,q)$, $q$ odd, which square to the identity.  Thus $\varepsilon(\chi) = 1$ for every real-valued irreducible $\chi$ of $GL(n,q)$ by the Frobenius-Schur theory.
\end{proof}

We extract the following identity which was obtained in the proof above.

\begin{cor} \label{cort} For any $q$, we have the formal identity
\[\frac{\prod_{i \geq 1} (1+u/q^i)^2}{\prod_{i \geq 1} (1-u^2/q^i)} = \sum_{n \geq 0} u^n q^{{n \choose 2}} \sum_{r=0}^n \frac{1}{\gamma_r \gamma_{n-r}}.\]
\end{cor}

Corollary \ref{cort} will be used in the next section, for the real character degree sums for the finite unitary groups.

\section{Real character degree sums for $U(n,q)$} \label{unitary}

We first establish some notation and results for polynomials, extending some of the notions used in Section \ref{GLnSum}.
Define the maps $\tilde{F}$ and $F$ on $\mathbb{F}_q^{\times}$ by $\tilde{F}(a) = a^{-q}$ and $F(a) = a^q$.  Let $[a]_{\tilde{F}}$ and $[a]_F$ denote the $\tilde{F}$ and $F$-orbits of $a \in \mathbb{F}_q^{\times}$, respectively.  For any $\tilde{F}$-orbit $[a]_{\tilde{F}}$ of size $d$, the \emph{U-irreducible} polynomial corresponding to it is the polynomial
$$\prod_{b \in [a]_{\tilde{F}}} (t - b) = (t - a)(t-a^{-q}) \cdots (t- a^{(-q)^{d-1}}).$$
Then each U-irreducible polynomial is in $\F_{q^2}[t]$.  It is known that if $d$ is odd, the U-irreducible polynomial is irreducible as a polynomial in $\F_{q^2}[t]$, and if $d$ is even, it is the product of two irreducible polynomials (an irreducible and its $\sim$-self-conjugate, as in \cite{FNP}), although we will not need these facts.

Let $\bar{N}(d,q)$ denote the number of U-irreducible polynomials of degree $d$.  As in \cite{TV}, since the group of fixed points of $\tilde{F}^m$ has cardinality $q^m - (-1)^m$, we have
$$ \sum_{r|m} r \bar{N}(r,q) = q^m - (-1)^m,$$
and it follows from M\"{o}bius inversion that we have
$$ \bar{N}(d,q) = \frac{1}{d} \sum_{r|d} \mu(r) (q^{d/r} - (-1)^{d/r}).$$
In particular, when $d>1$ is odd,
$$ \bar{N}(d,q) = \frac{1}{d} \sum_{r|d} \mu(r) (q^{d/r} + 1) = \frac{1}{d} \sum_{r|d} \mu(r) q^{d/r} = N(d,q),$$
where $N(d,q)$ is the number of degree $d$ monic irreducible polynomials over $\mathbb{F}_q$.

It is known that any monic $\sim$-self-conjugate polynomial in $\F_{q^2}[t]$ may be factored uniquely into a product of U-irreducible polynomials, and that the number of monic $\sim$-self-conjugate polynomials of degree $n$ in $\F_{q^2}[t]$ is $q^n + q^{n-1}$ (\cite[Lemma 1.3.11(a)]{FNP}).  Letting $\mathcal{U}$ denote the set of all $U$-irreducible polynomials, it follows that we have
\begin{eqnarray*}
\prod_{d \geq 1} (1 - w^d)^{-\bar{N}(d,q)} & = & \prod_{\phi \in \mathcal{U}} (1 + w^{\mathrm{deg}(\phi)} + w^{2\mathrm{deg}(\phi)} + \cdots ) \\
 & = & 1 + \sum_{n \geq 1} (q^n + q^{n-1}) w^n \\
& = & \frac{1+w}{1 - qw}.
\end{eqnarray*}

Now define the \emph{dual} of a U-irreducible polynomial as follows.  If $\phi$ is the U-irreducible polynomial corresponding to $[a]_{\tilde{F}}$, let the dual $\phi^*$ be defined as the U-irreducible polynomial corresponding to $[a^{-1}]_{\tilde{F}}$.  Let $\bar{N}^*(d,q)$ denote the number of self-dual U-irreducible polynomials of degree $d$ in $\F_{q^2}[t]$, and let $\bar{M}^*(d,q)$ denote the number of unordered pairs $\{ \phi, \phi^* \}$, where $\phi \in \mathcal{U}$, $\mathrm{deg}(\phi) = d$, and $\phi \neq \phi^*$.  By these definitions, we have
$$ \bar{N}(d,q) = \bar{N}^*(d,q) + 2 \bar{M}^* (d,q).$$
A key observation is that for any $a \in \bar{\F}_q^{\times}$, we have
$$ [a]_{\tilde{F}} \cup [a^{-1}]_{\tilde{F}} = [a]_F \cup [a^{-1}]_F,$$
where the orbits on the right correspond to an irreducible polynomial in $\F_q[t]$ and its dual.  From this, it follows that we have, for any $d \geq 1$,
\begin{equation} \label{starbar}
\bar{N}^*(2d,q) + \bar{M}^*(d,q) = N^*(2d,q) + M^*(d,q),
\end{equation}
and that $\bar{N}^*(d,q) = N^*(d,q)$ when $d$ is odd.  In particular, if we define, as in \cite{FNP}, $e=1$ if $q$ is even,
and $e=2$  if $q$ is odd, we have
\begin{equation} \label{starbarodd}
 \bar{N}^*(d,q) = \left\{ \begin{array}{ll} e & \text{ if } d = 1, \\ 0 & \text{ if } d \text{ is odd, } d > 1. \end{array} \right.
\end{equation}

Now, from (\ref{starbar}) and \cite[Lemma 1.3.17]{FNP}, it follows immediately that
\begin{equation} \label{Uprodlem1}
\prod_{d \geq 1} (1 - w^d)^{-\bar{N}^*(2d,q)} \prod_{d \geq 1} (1 - w^d)^{-\bar{M}^*(d,q)} = \frac{(1 - w)^e}{1 - qw},
\end{equation}
and
$$ \prod_{d \geq 1} (1 + w^d)^{-\bar{N}^*(2d,q)} \prod_{d \geq 1} (1 + w^d)^{-\bar{M}^*(d,q)} = \frac{(1 + w)^e(1-qw)}{1 - qw^2}.$$
We may now use the same techniques as in \cite{FNP} to get other identities.  In particular, from the facts that $\prod_{d \geq 1} (1 - w^d)^{-\bar{N}(d,q)} = (1 + w)/(1 - qw)$, $\bar{N}(d,q) = \bar{N}^*(d,q) + 2 \bar{M}^*(d,q)$, and (\ref{starbarodd}), we obtain
$$ \prod_{d \geq 1} (1 - w^{2d})^{-\bar{N}^*(2d,q)} \prod_{d \geq 1} (1 - w^d)^{-2\bar{M}^*(d,q)} = \frac{(1 + w)(1 - w)^e}{1-qw}.$$
From this and (\ref{Uprodlem1}), we have
\begin{equation} \label{Uprodlem2}
\prod_{d \geq 1} (1 + w^d)^{-\bar{N}^*(2d,q)} \prod_{d \geq 1} (1 - w^d)^{-\bar{M}^*(d,q)} = 1 + w.
\end{equation}

We are now able to compute the generating function for the sum of real character degrees for $U(n,q)$.

\begin{theorem} \label{degreesU} Let $e=1$ if the characteristic is even, and $e=2$ if the characteristic is odd. The sum of the degrees of the real characters of $U(n,q)$ is $(-1)^n (q^n - (-1)^n)\cdots (q+1)$ times the coefficient of $u^n$ in
$$ \prod_i \frac{(1 + u/(-q)^i)^e}{(1 + u^2/(-q)^{2i-1})} \prod_{i < j} (1 + u^2/(-q)^{i+j})^{-e+1} \frac{(1 - u^2/(-q)^{i+j})^e}{(1 + u^2/(-q)^{i+j-1})}.$$
\end{theorem}

\begin{proof} As in \cite{TV}, the irreducible characters of $U(n,q)$ may be parameterized by associating a partition $\lambda(\phi)$ to each U-irreducible polynomial $\phi \in \mathcal{U}$, such that $\sum_{\phi}d(\phi)|\lambda(\phi)| = n$.  It follows from \cite[Theorem 5.1]{TV} that the degree of the character corresponding to the parameters $\lambda(\phi)$ is
$$ (q^n - (-1)^n) \cdots (q + 1) \prod_{\phi} \frac{q^{d(\phi)n(\lambda(\phi)')}}{\prod_{b \in \lambda(\phi)} (q^{d(\phi)h(b)} - (-1)^{d(\phi)h(b)})}.$$
As in the $GL(n,q)$ case, use the identity $\sum_{b \in \lambda} h(b) = n(\lambda) + n(\lambda') + |\lambda|$, and factor appropriately to rewrite the above expression as
\begin{eqnarray*}
& & \frac{(q^n - (-1)^n) \cdots (q+1)}{q^n} \prod_{\phi} (-1)^{d(\phi)n(\lambda(\phi))} \frac{(-1/q)^{d(\phi)n(\lambda(\phi))}}{\prod_{b \in \lambda(\phi)} (1 - (-1/q)^{d(\phi)h(b)})} \\
& = & \frac{(q^n - (-1)^n) \cdots (q+1)}{q^n} \\
& & \cdot \prod_{\phi} (-1)^{d(\phi)n(\lambda(\phi))} s_{\lambda(\phi)} ( 1 , (-1/q)^{d(\phi)}, (-1/q)^{2d(\phi)}, \ldots ).
\end{eqnarray*}

A character of $U(n,q)$ is real-valued exactly when its parameters satisfy $\lambda(\phi) = \lambda(\phi^*)$ for every U-irreducible $\phi$, which follows from \cite[Lemma 3.1]{GV}.  Then the sum of the degrees of the real characters of $U(n,q)$ is $(q^n - (-1)^n)\cdots(q+1)/q^n$ times the coefficient of $u^n$ in

\begin{eqnarray*}
& & \prod_{d \geq 1}\left[ \sum_{\lambda} (-1)^{dn(\lambda)} u^{d|\lambda|} s_{\lambda}(1, (-q)^{-d}, (-q)^{-2d}, \cdots)\right]^{\bar{N}^*(d,q)}\\
& & \cdot \prod_{d \geq 1} \left[ \sum_{\lambda} u^{2d|\lambda|} s_{\lambda}(1, (-q)^{-d}, (-q)^{-2d}, \cdots)^2 \right]^{\bar{M}^*(d,q)},
\end{eqnarray*}
which is $(-1)^n (q^n - (-1)^n)\cdots(q+1)$ times the coefficient of $u^n$ in
\begin{eqnarray*}
& & \prod_{d \geq 1}\left[ \sum_{\lambda} (-1)^{dn(\lambda)} u^{d|\lambda|} s_{\lambda}((-q)^{-d}, (-q)^{-2d}, \cdots) \right]^{\bar{N}^*(d,q)}\\
& & \cdot \prod_{d \geq 1} \left[ \sum_{\lambda} u^{2d|\lambda|} s_{\lambda}((-q)^{-d}, (-q)^{-2d}, \cdots)^2\right]^{\bar{M}^*(d,q)}\\
& = & \prod_{d \geq 1}\left[ \sum_{\lambda} (-1)^{dn(\lambda)} s_{\lambda}(u^d/(-q)^d, u^d/(-q)^{2d}, \cdots) \right]^{\bar{N}^*(d,q)}\\
& & \cdot \prod_{d \geq 1} \left[ \sum_{\lambda} s_{\lambda}(u^d/(-q)^d, u^d/(-q)^{2d}, \cdots)^2 \right]^{\bar{M}^*(d,q)}.
\end{eqnarray*}
Note that we have
\begin{eqnarray*}
& & \prod_{d \geq 1}\left[ \sum_{\lambda} (-1)^{dn(\lambda)} s_{\lambda}(u^d/(-q)^d, u^d/(-q)^{2d}, \cdots) \right]^{\bar{N}^*(d,q)} \\
 & = & \prod_{d \text{ odd}} \left[\sum_{\lambda} (-1)^{n(\lambda)} s_{\lambda}(u^d/(-q)^d, u^d/(-q)^{2d}, \ldots)\right]^{\bar{N}^*(d,q)} \\
& & \cdot \prod_{d \text{ even}} \left[\sum_{\lambda} s_{\lambda}(u^d/(-q)^d, u^d/(-q)^{2d}, \ldots)\right]^{\bar{N}^*(d,q)}
\end{eqnarray*}

Next, we use the two identities for Schur functions \eqref{Schur1} and \eqref{Schur2} as in the $GL(n,q)$ case, but we also must use the identity from \cite[I.5, Ex. 6]{Mac},
$$ \sum_{\lambda} (-1)^{n(\lambda)} s_{\lambda} = \prod_i (1 - x_i)^{-1} \prod_{i < j} (1 + x_i x_j)^{-1}.$$
Applying these identities gives that the sum of the real character degrees of $U(n,q)$ is $(-1)^n (q^n - (-1)^n) \cdots (q+1)$ times the coefficient of $u^n$ in
\begin{eqnarray*}
& & \tilde{A}(u) \\
& : =  & \prod_d \left[ \prod_i (1 - u^d/(-q)^{id})^{-1} \prod_{i < j} (1 - (-1)^d u^{2d}/(-q)^{(i+j)d})^{-1} \right]^{\bar{N}^*(d,q)} \\
& & \cdot \prod_d \left[ \prod_{i,j} (1 - u^{2d}/(-q)^{(i+j)d})^{-1} \right]^{\bar{M}^*(d,q)}.
\end{eqnarray*}

Imitating the calculation for $GL(n,q)$, write $\tilde{A}(u) = \tilde{B}(u) \tilde{C}(u)$, where

$$ \tilde{B}(u) := \prod_i \prod_d (1 - u^d/(-q)^{id})^{-\bar{N}^*(d,q)} (1 - u^{2d}/(-q)^{2id})^{-\bar{M}^*(d,q)}$$
$$ \tilde{C}(u) := \prod_{i < j} \prod_d (1 - (-1)^d u^{2d}/(-q)^{(i+j)d})^{-\bar{N}^*(d,q)} (1 - u^{2d}/(-q)^{(i+j)d})^{-2\bar{M}^*(d,q)}$$
Now, by (\ref{starbarodd}) and (\ref{Uprodlem1}), we have
\begin{eqnarray*}
\tilde{B}(u) & = & \prod_i (1 - u/(-q)^i)^{-e} \prod_d (1 - u^{2d}/(-q)^{2id})^{-\bar{N}^*(2d,q)} \\
& & \cdot (1 - u^{2d}/(-q)^{2id})^{-\bar{M}^*(d,q)} \\
& = & \prod_i (1 - u/(-q)^i)^{-e} \frac{(1 - u^2/(-q)^{2i})^e}{(1 + u^2/(-q)^{2i-1})} \\
& = & \prod_i \frac{(1 + u/(-q)^i)^e}{(1 + u^2/(-q)^{2i-1})}.
\end{eqnarray*}

By (\ref{starbarodd}), (\ref{Uprodlem1}) and (\ref{Uprodlem2}), we compute that $\tilde{C}(u)$ is equal to
\begin{eqnarray*}
& & \prod_{i < j} \left( 1 + \frac{u^2}{(-q)^{i+j}}\right)^{-e} \prod_d \left( 1 - \frac{u^{4d}}{(-q)^{2d(i+j)}} \right)^{-\bar{N}^*(2d,q)} \\
& & \cdot \left( 1 - \frac{u^{2d}}{(-q)^{d(i+j)}} \right)^{-2\bar{M}^*(d,q)} \\
& = & \prod_{i < j} \left( 1 + \frac{u^2}{(-q)^{i+j}}\right)^{-e} \\
& & \cdot \prod_d  \left( 1 - \frac{u^{2d}}{(-q)^{d(i+j)}} \right)^{-\bar{N}^*(2d,q)} \left( 1 - \frac{u^{2d}}{(-q)^{d(i+j)}} \right)^{-\bar{M}^*(d,q)} \\
& & \cdot \prod_d  \left( 1 + \frac{u^{2d}}{(-q)^{d(i+j)}} \right)^{-\bar{N}^*(2d,q)} \left( 1 - \frac{u^{2d}}{(-q)^{d(i+j)}} \right)^{-\bar{M}^*(d,q)}\\
& = & \prod_{i < j} ( 1 + u^2/(-q)^{i+j})^{-e} \frac{(1 - u^2/(-q)^{i+j})^e}{(1 + u^2/(-q)^{i+j-1})} (1 + u^2/(-q)^{i+j}) \\
& = & \prod_{i < j} (1 + u^2/(-q)^{i+j})^{-e+1} \frac{(1 - u^2/(-q)^{i+j})^e}{(1 + u^2/(-q)^{i+j-1})}.
\end{eqnarray*}
We now have
\begin{align*}
\tilde{A}(u) & =  \tilde{B}(u) \tilde{C}(u) \\
&=\prod_i \frac{(1 + u/(-q)^i)^e}{(1 + u^2/(-q)^{2i-1})} \prod_{i < j} (1 + u^2/(-q)^{i+j})^{-e+1} \frac{(1 - u^2/(-q)^{i+j})^e}{(1 + u^2/(-q)^{i+j-1})},
\end{align*}
completing the proof.
\end{proof}

In order to expand the generating function computed above as a series in $u$, we need some notation to state a key result of Warnaar \cite{War}.

For any partition $\lambda$, let $\lambda_{\text{e}}$ and $\lambda_{\text{o}}$ denote the partitions consisting of only the even parts and odd parts of $\lambda$, respectively.  Let $\ell(\lambda)$ be the number of parts of $\lambda$, so that $\ell(\lambda_{\text{o}})$ is the number of odd parts of $\lambda$.  For any positive integer $j$, let $m_j(\lambda)$ be the multiplicity of $j$ in $\lambda$.  For example, if $\lambda = (7, 4, 4, 3, 3, 2, 1, 1, 1)$, then $\lambda_{\text{e}} = (4, 4, 2)$, $\lambda_{\text{o}} = (7, 3, 3, 1, 1, 1)$, $\ell(\lambda) = 9$, $\ell(\lambda_{\text{o}}) = 6$, $m_7 (\lambda) = m_2(\lambda) = 1$, $m_4(\lambda) = m_3 (\lambda) = 2$, and $m_1(\lambda) = 3$.  We will also denote a partition $\mu$ by the notation $(1^{m_1(\mu)} 2^{m_2(\mu)} 3^{m_3(\mu)} \cdots )$, so that the $\lambda$ in the example may be written as $\lambda = (7^1 4^2 3^2 2^1 1^3)$.

Given the variables $x = \{x_1, x_2, \ldots \}$, an indeterminate $t$, and a partition $\lambda$, let $P_{\lambda}(x ;t)$ denote the Hall-Littlewood symmetric function (see \cite[Chapter III]{Mac}).  Given the single variable $z$, and $m \geq 0$ an integer, let $H_m(z ; t)$ denote the Rogers-Szeg\H{o} polynomial (see \cite[Chapter 3, Examples 3-9]{A}), defined as
$$ H_m(z ; t) = \sum_{j=0}^m \qb{m}{j}_t z^j,$$
where, as in Section \ref{GLnSum}, $\qb{m}{j}_t$ is defined to be the following polynomial in $t$:
$$ \qb{m}{j}_t = \frac{(t^m - 1) \cdots (t-1)}{(t^j - 1) \cdots (t-1) (t^{m-j} - 1) \cdots (t-1)}.$$
Given any partition $\lambda$, define the more general Rogers-Szeg\H{o} polynomial $h_{\lambda}(z ; t)$ by
$$ h_{\lambda}(z ; t) = \prod_{i \geq 1} H_{m_i(\lambda)}(z ;t),$$
so that $h_{(1^m)}(z ;t) = H_m(z ; t)$.

We may now state the following identity \cite[Theorem 1.1]{War}.

\begin{theorem}[Warnaar] \label{WarId} The following identity holds, where the sum is over all partitions:
$$ \sum_{\lambda} a^{\ell(\lambda_{\text{o}})} h_{\lambda_{\text{e}}}(ab ; t) h_{\lambda_{\text{o}}}(b/a ; t) P_{\lambda}(x ; t) = \prod_{i \geq 1} \frac{(1 + ax_i)(1 + bx_i)}{(1 - x_i)(1 + x_i)} \prod_{i < j} \frac{1 - tx_i x_j}{1 - x_i x_j}.$$
\end{theorem}

The following specialization of Theorem \ref{WarId} to the case $b=0$ will also be useful for us \cite[Corollary 1.3]{War}.

\begin{cor}[Warnaar] \label{WarCor} The following identity holds, where the sum is over all partitions:
$$ \sum_{\lambda} a^{\ell(\lambda_{\text{o}})} P_{\lambda}(x ; t) = \prod_{i \geq 1} \frac{1 + ax_i}{1 - x_i^2} \prod_{i < j} \frac{1 - tx_i x_j}{1 - x_i x_j}.$$
\end{cor}

\subsection{Characteristic two}
From here on, we define
\[ \omega_n = |U(n,q)| = q^{n(n-1)/2} \prod_{i=1}^n \left(q^i - (-1)^i\right) = q^{n^2} \prod_{i=1}^n \left(1 - (-1/q)^i \right), \] and set $\omega_0 = 1$.

We first concentrate on the case that $q$ is even.  We may compute the number of involutions in $U(n,q)$ as follows.

\begin{prop} \label{involUeven} Let $q$ be even. Then the number of involutions in $U(n,q)$ is equal to
\[ \sum_{r=0}^{\lfloor n/2 \rfloor} \frac{\omega_n}{q^{r(2n-3r)} \omega_r \omega_{n-2r}},\]
and also to $(-1)^{n+{n \choose 2}} (q^n - (-1)^n) \cdots (q+1)$ times the coefficient of $u^n$ in:
\[ \prod_i \frac{1+u/(-q)^i}{1-u^2/(-q)^i}.\]
\end{prop}
\begin{proof} Since $q$ is even, an involution in $U(n,q)$ in this case is a unipotent element of type $(2^r 1^{n-2r})$, that is, its elementary divisors are $(t-1)^2$ with multiplicity $r$ and $t-1$ with multiplicity $n-2r$.  It follows from \cite{Wall}, for example, that the centralizer of such an element in $U(n,q)$ has order
$$ q^{r(2n-3r)} \omega_r \omega_{n-2r}.$$
So, the total number of involutions is the sum of the indices of these centralizers in $U(n,q)$, for $r = 0, \ldots, \lfloor n/2 \rfloor$, so that the total number of involutions is
\begin{eqnarray*}
& & \sum_{r=0}^{\lfloor n/2 \rfloor} \frac{\omega_n}{q^{r(2n-3r)} \omega_r \omega_{n-2r}} \\
& = & (-1)^n (q^n - (-1)^n) \cdots (q+1) \sum_{r=0}^{\lfloor n/2 \rfloor} \frac{(-1)^n q^{n(n-1)/2}}{q^{r(2n-3r)} \omega_r \omega_{n-2r}} \\
 & =: & (-1)^n (q^n - (-1)^n) \cdots (q+1) J_n(q).
\end{eqnarray*}
Now note that we have
$$ J_n(-q) = (-1)^{n(n-1)/2} \sum_{r=0}^{\lfloor n/2 \rfloor} \frac{q^{n(n-1)/2}}{q^{r(2n-3r)} \gamma_r \gamma_{n-2r}},$$
From Corollary \ref{iden}, we have $J_n(-q)$ is $(-1)^{n(n-1)/2}$ times the coefficient of $u^n$ in
$$ \prod_i \frac{1 + u/q^i}{1 - u^2/q^i}.$$
Making the substitution of $-q$ for $q$, then, we have that the number of involutions in $U(n,q)$ is
$$(-1)^{n(n-1)/2} (-1)^n (q^n - (-1)^n) \cdots (q+1)$$ times the coefficient of $u^n$ in
$$ \prod_i \frac{1 + u/(-q)^i}{1 - u^2/(-q)^i},$$ which completes the proof.
\end{proof}

By (\ref{invol}), Proposition \ref{involUeven} gives a generating function for
$$\sum_{\chi \in \Irr(G) \atop{ \varepsilon(\chi) = 1}} \chi(1)- \sum_{\chi \in \Irr(G) \atop {\varepsilon(\chi) = -1}} \chi(1),$$
where $G = U(n,q)$, $q$ even.  Since Theorem \ref{degreesU} gives a generating function for
$$\sum_{\chi \in \Irr(G) \atop{ \varepsilon(\chi) = 1}} \chi(1)+\sum_{\chi \in \Irr(G) \atop {\varepsilon(\chi) = -1}} \chi(1),$$
then we may immediately obtain the following (with $e=1$ in Theorem \ref{degreesU}).

\begin{cor} \label{GenFnEpEven} Let $q$ be even.  The sum of the character degrees of $U(n,q)$ with Frobenius-Schur indicator $\pm 1$ is equal to $(-1)^n (q^n - (-1)^n) \cdots (q+1)$ times the coefficient of $u^n$ in
\[ \frac{1}{2} \prod_i \frac{1 + u/(-q)^i}{1 + u^2/(-q)^{2i-1}} \prod_{i < j} \frac{1 - u^2/(-q)^{i+j}}{1+u^2/(-q)^{i+j-1}} \pm \frac{(-1)^{{n \choose 2}}}{2} \prod_i \frac{1 + u/(-q)^i}{1-u^2/(-q)^i} .\]



\end{cor}

We now expand the generating function in Theorem \ref{degreesU} for the case that $q$ is even.

\begin{theorem} \label{UnSumEven} Let $q$ be even.  Then the sum of the degrees of the real-valued characters of $U(n,q)$ is
$$ (q^n - (-1)^n) \cdots (q+1) \sum_{|\lambda| = n} q^{-(\ell(\lambda_{\text{o}}) +n)/2} P_{\lambda}(1, (-q)^{-1}, (-q)^{-2}, \ldots ; q^{-1}).$$
\end{theorem}
\begin{proof} By Theorem \ref{degreesU}, with $e=1$, the sum of the degrees of the real-valued characters of $U(n,q)$ is $(-1)^n (q^n - (-1)^n) \cdots (q+1)$ times the coefficient of $u^n$ in
$$ \prod_i \frac{ 1 + u/(-q)^i}{1 + u^2/(-q)^{2i-1}} \prod_{i < j} \frac{ 1 - u^2/(-q)^{i+j}}{1 + u^2/(-q)^{i + j -1}}.$$
To expand this, we apply Corollary \ref{WarCor}, with the substitutions
$$ x_i = - u q^{1/2} (-q)^{-i} = u q^{-1/2} (-q)^{-(i-1)}, \quad a = -q^{-1/2}, \quad t = q^{-1}.$$
Then we have
\begin{align*}
\prod_{i \geq 1} &\frac{1 + ax_i}{1 - x_i^2} \prod_{i < j} \frac{1 - tx_i x_j}{1 - x_i x_j}  = \prod_i \frac{ 1 + u/(-q)^i}{1 + u^2/(-q)^{2i-1}} \prod_{i < j} \frac{ 1 - u^2/(-q)^{i+j}}{1 + u^2/(-q)^{i + j -1}} \\
 & = \sum_{\lambda} (-q^{-1/2})^{\ell(\lambda_{\text{o}})} P_{\lambda} (u q^{-1/2}, u q^{-1/2} (-q)^{-1}, u q^{-1/2} (-q)^{-2},\ldots; q^{-1}) \\
 & = \sum_{\lambda} (-1)^{\ell(\lambda_{\text{o}})} q^{-(\ell(\lambda_{\text{o}}) + |\lambda|)/2} u^{|\lambda|} P_{\lambda} (1, (-q)^{-1}, (-q)^{-2}, \ldots; q^{-1}).
\end{align*}
The coefficient of $u^n$ in this series is
$$ \sum_{|\lambda| = n} (-1)^{\ell(\lambda_{\text{o}})} q^{-(\ell(\lambda_{\text{o}}) +n)/2} P_{\lambda}(1, (-q)^{-1}, (-q)^{-2}, \ldots ; q^{-1}).$$
The result now follows from the observation that $|\lambda|$ and $\ell(\lambda_{\text{o}})$ have the same parity.
\end{proof}

We can immediately get the following result by applying Theorem \ref{UnSumEven} and Proposition \ref{involUeven}.

\begin{cor} \label{UnSumEven+-}
Let $q$ be even.  Then the sum of the degrees of the characters of $U(n,q)$ with Frobenius-Schur indicator $\pm 1$ is
\begin{align*}
\frac{1}{2} (q^n - (-1)^n) \cdots (q+1) &\sum_{|\lambda| = n}  q^{-(\ell(\lambda_{\text{o}}) +n)/2} P_{\lambda}(1, (-q)^{-1}, (-q)^{-2}, \ldots ; q^{-1}) \\
&\pm \frac{1}{2} \sum_{r=0}^{\lfloor n/2 \rfloor} \frac{\omega_n}{q^{r(2n-3r)} \omega_r \omega_{n-2r}}.
\end{align*}
\end{cor}

We can express the quantities in Corollary \ref{UnSumEven+-} in yet another way by applying Corollary \ref{GenFnEpEven}.  While the expression we obtain seems to be more complicated, we give an example in which it makes calculation somewhat simpler.

\begin{cor} \label{GenFnEvenAlt} Let $q$ be even.  Then the sum of the degrees of the characters of $U(n,q)$ with Frobenius-Schur indicator $\pm 1$ is $(-1)^n (q^n - (-1)^n) \cdots (q+1)$ times
\begin{align*}
\frac{1 \pm (-1)^{n \choose 2}}{2} & \sum_{r = 0}^{\lfloor n/2 \rfloor} \frac{(-1)^{n + {n \choose 2}} q^{{n \choose 2}}}{q^{r(2n - 3r)} \omega_r \omega_{n-2r}} \\
& + \frac{1}{2} \sum_{k=1}^{\lfloor n/2 \rfloor} \left[ \left(\sum_{\ell(\lambda_{\text{o}}) + |\lambda| = 2k} q^{-k} P_{\lambda} (1, (-q)^{-1}, (-q)^{-2}, \ldots ; q^{-1}) \right) \right. \\
& \quad \cdot \left. \left( \sum_{s = 0}^{\lfloor (n-2k)/2 \rfloor} \frac{(-1)^{n - 2k + { n - 2k \choose 2}} q^{n - 2k \choose 2}}{q^{s(2n-4k-3s)} \omega_{s} \omega_{n-2k-2s}} \right) \right].
\end{align*}
\end{cor}
\begin{proof} By Corollary \ref{GenFnEpEven}, the desired sum is $(-1)^n (q^n - (-1)^n) \cdots (q+1)$ times the coefficient of $u^n$ in
$$\frac{1}{2} \prod_i \frac{1 + u/(-q)^i}{1 + u^2/(-q)^{2i-1}} \prod_{i < j} \frac{1 - u^2/(-q)^{i+j}}{1+u^2/(-q)^{i+j-1}} \pm \frac{(-1)^{{n \choose 2}}}{2} \prod_i \frac{1 + u/(-q)^i}{1-u^2/(-q)^i}$$
\begin{align}
= \frac{1}{2} \prod_i & \frac{1 + u/(-q)^i}{1 - u^2/(-q)^i}  \label{Corprod}\\
&  \cdot  \left( \prod_i \frac{1 - u^2/(-q)^i}{1 + u^2/(-q)^{2i-1}} \prod_{i < j} \frac{1 - u^2/(-q)^{i+j}}{1 + u^2/(-q)^{i + j - 1}} \pm (-1)^{{n \choose 2}} \right). \notag
\end{align}
We now apply Corollary \ref{WarCor} with the substitutions $x_i = -u q^{1/2} (-q)^{-i}$, $a = uq^{-1/2}$, and $t = q^{-1}$ to obtain
\begin{align*}
\prod_i &\frac{1 - u^2/(-q)^i}{1 + u^2/(-q)^{2i-1}} \prod_{i < j} \frac{1 - u^2/(-q)^{i+j}}{1 + u^2/(-q)^{i + j - 1}} \\
&= \sum_{\lambda} u^{\ell(\lambda_{\text{o}}) + |\lambda|} q^{-(\ell(\lambda_{\text{o}}) + |\lambda|)/2} P_{\lambda} (1 , (-q)^{-1}, (-q)^{-2}, \ldots; q^{-1}).
\end{align*}
Noting that $\ell(\lambda_{\text{o}} )+ |\lambda|$ is always even, this gives that the coefficient of $u^{2k}$ in this product is
$$\sum_{\ell(\lambda_{\text{o}}) + |\lambda| = 2k} q^{-k} P_{\lambda} (1 , (-q)^{-1}, (-q)^{-2}, \ldots; q^{-1})$$
if $k \geq 0$ and the coefficient of $u^{2k+1}$, $k \geq 0$, is $0$.  By Proposition \ref{involUeven}, the coefficient of $u^m$ in $\prod_i \frac{1 + u/(-q)^i}{1-u^2/(-q)^i}$ is
$$ \sum_{r=0}^{\lfloor m/2 \rfloor} \frac{ (-1)^{m + {m \choose 2}} q^{m \choose 2}}{q^{r(2m-3r)} \omega_r \omega_{m-2r}}.$$
Substituting these to find the coefficient of $u^n$ in \eqref{Corprod} gives the result.
\end{proof}

While there does exist a nice expression for the value of $P_{\lambda}(1, t, t^2, \ldots; t)$ (see \cite[III.2, Example 1]{Mac}), the authors are unaware of such an expression for $P_{\lambda}(1, t, t^2, \ldots; -t)$.  So, we are unable to further simplify the expression in Theorem \ref{UnSumEven}.  However, there are nice evaluations of both $P_{(1^m)}(x ; t)$ and $P_{(m)}(x ; t)$ which we may apply.  In particular, $P_{(1^m)}(x ; t) = e_m(x)$, where $e_m(x)$ is the elementary symmetric function, by \cite[III.2, Equation (2.8)]{Mac}.  Then, by \cite[I.2, Example 4]{Mac}, we have
\begin{align} \label{P1^m}
P_{(1^m)}(1, (-q)^{-1}, (-q)^{-2}, \ldots; q^{-1}) & = e_m(1, (-q)^{-1}, (-q)^{-2}, \ldots)  \notag\\
 = \frac{(-q)^{-m(m-1)/2}}{(1 + (1/q)) \cdots (1  - (-1/q)^m)} & = \frac{(-1)^{m(m-1)/2} q^m}{(q+1) \cdots (q^m - (-1)^m)}.
\end{align}
By \cite[III.2, Equation (2.10)]{Mac}, $P_{(m)}(x ;t)$ is $(1-t)^{-1}$ times the coefficient of $u^m$ in $\prod_i \frac{1-x_itu}{1 - x_i u}$.  So, $P_{(m)}(1, (-q)^{-1}, (-q)^{-2}, \ldots; q^{-1})$ is $(1 - (1/q))^{-1}$ times the coefficient of $u^m$ in
$$ \frac{\prod_{i \geq 1} (1 + u/(-q)^i)}{\prod_{i \geq 1} (1 - u/(-q)^{i-1})}.$$
By Lemma \ref{eul},
$$\prod_{i \geq 1} (1 + u/(-q)^i) = \sum_{r \geq 0} \frac{(-1)^{r(r+1)/2} u^r}{(q+1) \cdots (q^r - (-1)^r)},$$
and
$$ \prod_{i \geq 1} (1 - u/(-q)^{i-1})^{-1} = \sum_{r \geq 0} \frac{q^{r(r+1)/2} u^r}{(q+1) \cdots (q^r - (-1)^r)}.$$
So, we have
\begin{align} \label{Pm}
P_{(m)} & (1, (-q)^{-1}, (-q)^{-2}, \ldots ; q^{-1}) \notag \\
& = \frac{q}{q-1} \sum_{r = 0}^m \frac{(-1)^{r(r+1)/2} q^{(m-r)(m-r+1)/2}}{(q+1) \cdots (q^r - (-1)^r) (q+1) \cdots (q^{m-r} - (-1)^{m-r})}.
\end{align}

For example, to compute the sum of the real character degrees of $U(2,q)$, for $q$ even, then by Theorem \ref{UnSumEven}, the only partitions in the sum are $\lambda = (2)$ and $\lambda = (1^2)$, so this sum is
\begin{align*}
(q^2 - 1)(q+1) \left[q^{-1} \right. & P_{(2)} ( 1, (-q)^{-1}, (-q)^{-2}, \ldots; q^{-1})  \\
& \left. + q^{-2} P_{(1^2)} ( 1, (-q)^{-1}, (-q)^{-2}, \ldots; q^{-1}) \right].
\end{align*}
Directly from \eqref{Pm} and \eqref{P1^m}, we may compute that
$$P_{(2)}(1, (-q)^{-1}, (-q)^{-2}, \ldots; q^{-1}) = \frac{q(q^2+1)}{(q+1)(q^2 -1)}, \quad \text{ and }$$
$$P_{(1^2)}(1, (-q)^{-1}, (-q)^{-2}, \ldots; q^{-1}) = \frac{-q^2}{(q+1)(q^2 - 1)},$$
from which it follows that the sum of the real character degrees of $U(2,q)$, $q$ even, is $q^2$.  By Proposition \ref{involUeven}, this is also the number of involutions in $U(2,q)$, meaning that $\varepsilon(\chi) = 1$ for every real-valued irreducible character $\chi$ of $U(2,q)$, $q$ even.  This is mentioned in the last paragraph of a paper of Gow \cite{Gow}, and this is also implied by a result of Ohmori \cite[Theorem 7(ii)]{Ohm81} (every character of $U(2,q)$ is either regular or semisimple, and the result states that such real-valued characters satisfy $\varepsilon(\chi) = 1$ when $q$ is even).

Now consider the sum of the degrees of the characters of $U(3,q)$, $q$ even, with Frobenius-Schur indicator 1.  Note that if we apply Corollary \ref{UnSumEven+-}, we need the value of $P_{\lambda}(1, (-q)^{-1}, (-q)^{-2}, \ldots ; q^{-1})$ for $\lambda = (3), (2,1),$ and $(1^3)$.  However, using Corollary \ref{GenFnEvenAlt}, we only need this value for $\lambda =(1)$ and $(2)$.  That is, Corollary \ref{GenFnEvenAlt} makes this calculation a bit easier, and from that result the sum of the degrees of the characters of $U(3,q)$ with Frobenius-Schur indicator 1 is equal to $-(q^3 +1)(q^2 - 1)(q+1)$ times
$$ \frac{1}{2} \left(\sum_{\ell(\lambda_{\text{o}}) + |\lambda| = 2} q^{-1} P_{\lambda}(1, (-q)^{-1}, (-q)^{-2}, \ldots ; q^{-1}) \right) \left( \frac{-1}{q+1} \right)$$
$$ = \frac{ P_{(1)}(1, (-q)^{-1}, (-q)^{-2}, \ldots; q^{-1}) + P_{(2)}(1, (-q)^{-1}, (-q)^{-2}, \ldots ; q^{-1})}{-2q(q+1)}$$
$$ = \frac{-1}{2q (q+1)} \left( \frac{q}{q+1} + \frac{q(q^2 + 1)}{(q+1)(q^2 - 1)} \right) = \frac{-q^2}{(q+1)^2 (q^2 - 1)}.$$
 This gives that this character degree sum is $q^4 - q^3 + q^2$.  By Proposition \ref{involUeven}, this character degree sum, minus the sum of the degrees of characters with Frobenius-Schur indicator $-1$, is $q^4 -q^3 +q$.  So, the sum of the degrees of characters with Frobenius-Schur indicator $-1$ is $q^2 -q$.  By the result of Ohmori \cite[Theorem 7(ii)]{Ohm81}, every real-valued semisimple and regular character of $U(3,q)$ has Frobenius-Schur indicator $1$, and the only other real-valued character is the unique cuspidal unipotent character, which has degree $q^2 - q$, and this character must thus be the unique character of $U(3,q)$ with Frobenius-Schur indicator $-1$.  The fact that this unipotent character has Frobenius-Schur indicator $-1$ is also consistent with the general result for unipotent characters \cite{Ohm}.

\subsection{Odd characteristic}

As in the case that $q$ is even, we begin by counting the involutions in $U(n,q)$ when $q$ is odd.

\begin{prop}  \label{involUodd} Let $q$ be odd. Then the number of involutions in $U(n,q)$ is equal to
$$\sum_{r=0}^n \frac{\omega_n}{\omega_r \omega_{n-r}}$$
and also to $(-1)^{n+{n \choose 2}} (q^n - (-1)^n) \cdots (q+1)$ times the coefficient of $u^n$ in:
\[ \prod_i \frac{(1+u/(-q)^i)^2}{1-u^2/(-q)^i}.\]
\end{prop}

\begin{proof} For $q$ odd, an involution in $U(n,q)$ has eigenvalues $1$ and $-1$, with each Jordan block having size $1$.  If such an element has eigenvalue $1$ with multiplicity $r$ and $-1$ with multiplicity $n-r$, then it is conjugate over an algebraic closure to a diagonal matrix.  Such an element has centralizer isomorphic to $U(r,q) \times U(n-r,q)$, by \cite{Wall}.  Thus, the total number of involutions in $U(n,q)$, $q$ odd is
\begin{align*}
\sum_{r=0}^n \frac{\omega_n}{\omega_r \omega_{n-r}} & = (-1)^n (q^n - (-1)^n) \cdots (q+1) \sum_{r=0}^n \frac{(-1)^n q^{n(n-1)/2}}{\omega_r \omega_{n-r}} \\
       &=: (-1)^n (q^n - (-1)^n) \cdots (q+1) I_n(q).
\end{align*}
Then one checks that we have
$$ I_n(-q) =(-1)^{n(n-1)/2} \sum_{r=0}^n \frac{q^{n(n-1)/2}}{\gamma_r \gamma_{n-r}}.$$
Now, we know from Corollary \ref{cort} that $\sum_{r=0}^{n} \frac{q^{n(n-1)/2}}{\gamma_r \gamma_{n-r}}$ is the coefficient of $u^n$ in
$$ \prod_i \frac{(1 + u/q^i)^2}{1-u^2/q^i},$$ and by substituting $-q$ for $q$, the result follows.
\end{proof}

By precisely the same argument as in the $q$ even case, we obtain the following result.

\begin{cor} \label{GenFnEpOdd} Let $q$ be odd.  The sum of the degrees of $U(n,q)$ with Frobenius-Schur indicator $\pm 1$ is equal to $(-1)^n (q^n - (-1)^n) \cdots (q+1)$ times the coefficient of $u^n$ in
\begin{eqnarray*}
& & \frac{1}{2}  \prod_i \frac{(1 + u/(-q)^i)^2}{1 + u^2/(-q)^{2i-1}} \prod_{i < j} \frac{(1-u^2/(-q)^{i+j})^2}{(1 + u^2/(-q)^{i+j})(1+u^2/(-q)^{i+j-1})} \\
& & \pm \frac{(-1)^{{n \choose 2}}}{2} \prod_i \frac{(1 + u/(-q)^i)^2}{1-u^2/(-q)^i}.
\end{eqnarray*}
\end{cor}

We now expand the generating function from Theorem \ref{degreesU} when $q$ is odd.  In this case, we need the more general result Theorem \ref{WarId} rather than Corollary \ref{WarCor} as in the $q$ even case.  We need just a bit more notation.  For any partition $\nu$, let $\nu'$ denote the conjugate partition of $\nu$.  A partition is said to be {\it even} if all of its parts are even.  For any $c,d$, and any integer $m \geq 1$, we define $(c ; d)_m$ by
$$ (c ; d)_m = (1 - c)(1 - cd) \cdots (1 - cd^{m-1}).$$ We set $(c;d)_0=1$. Finally, we note that $P_{\lambda}(x ; -1)$ is a symmetric function studied in \cite[III.8]{Mac}.

In the following, we give two expressions for the sum of the real character degrees for $U(n,q)$ with $q$ odd.  While the first is a bit more notationally manageable, the second could be considered computationally advantageous as it requires fewer special values of Rogers-Szeg\H{o} polynomials which have no convenient factorization.

\begin{theorem} \label{UnSumOdd} Let $q$ be odd.  Then the sum of the degrees of the real-valued characters of $U(n,q)$ is $(-1)^n (q^n - (-1)^n) \cdots (q+1)$ times
\begin{align*}
\sum_{|\lambda|+|\nu| = n \atop{ \nu' \text{ even}}} &  (-1)^{|\nu|/2 + \ell(\lambda_{\text{o}})} q^{-|\nu| - (\ell(\lambda_{\text{o}}) + |\lambda|)/2} 2^{\ell(\nu)/2} h_{\lambda_{\text{e}}}(q^{-1} ; q^{-1}) h_{\lambda_{\text{o}}}(1 ; q^{-1}) \\
& \cdot  P_{\lambda}(1, (-q)^{-1}, (-q)^{-2}, \ldots; q^{-1}) P_{\nu} (1, (-q)^{-1}, (-q)^{-2}, \ldots;-1) \\
= \sum_{|\lambda| + |\nu| = n \atop{ (\lambda_{\text{o}})', (\nu_{\text{e}})' \text{ even}}} & (-1)^{(\ell(\lambda_{\text{o}}) + \ell(\nu_{\text{o}}) + |\nu|)/2} q^{-|\nu| - (\ell(\lambda_{\text{o}}) + |\lambda|)/2} \left(\prod_i 2^{\lceil m_i(\nu)/2 \rceil} \right) \\
& \cdot h_{\lambda_{\text{e}}} (q^{-1} ; q^{-1}) \left( \prod_i (q^{-1} ; q^{-2})_{m_i(\lambda_{\text{o}})/2} \right) \\
& \cdot  P_{\lambda}(1, (-q)^{-1}, (-q)^{-2}, \ldots; q^{-1}) P_{\nu} (1, (-q)^{-1}, (-q)^{-2}, \ldots;-1)
\end{align*}

\end{theorem}
\begin{proof}
By Theorem \ref{degreesU}, with $e=2$, the sum of the degrees of the real-valued characters of $U(n,q)$ is $(-1)^n (q^n - (-1)^n) \cdots (q+1)$ times the coefficient of $u^n$ in
\begin{align}
\prod_i & \frac{(1 + u/(-q)^i)^2}{1 + u^2/(-q)^{2i-1}} \prod_{i < j} \frac{(1 - u^2/(-q)^{i+j})^2}{(1 + u^2/(-q)^{i+j})(1 + u^2/(-q)^{i+j-1})} \notag\\
& = \left( \prod_i \frac{(1 + u/(-q)^i)^2}{1 + u^2/(-q)^{2i-1}} \prod_{i < j} \frac{1 - u^2/(-q)^{i+j}}{1 + u^2/(-q)^{i+j-1}} \right) \left( \prod_{i < j} \frac{1 - u^2/(-q)^{i+j}}{1 + u^2/(-q)^{i+j}} \right)  \label{1stExp}
\end{align}
From Theorem \ref{WarId}, with the substitutions
$$ a = b = -q^{-1/2}, \quad x_i = -u q^{1/2} (-q)^{-i}, \quad  t = q^{-1},$$
we have
\begin{align*}
\prod_{i} & \frac{(1 + ax_i)^2}{1 - x_i^2} \prod_{i < j} \frac{1 - tx_i x_j}{1 - x_i x_j} = \prod_i \frac{(1 + u/(-q)^i)^2}{1 + u^2/(-q)^{2i-1}} \prod_{i < j} \frac{1 - u^2/(-q)^{i+j}}{1 + u^2/(-q)^{i+j-1}} \\
& = \sum_{\lambda} (-1)^{\ell(\lambda_{\text{o}})} q^{-(\ell(\lambda_{\text{o}})+ |\lambda|)/2} h_{\lambda_{\text{e}}} (q^{-1} ; q^{-1}) h_{\lambda_{\text{o}}}(1 ; q^{-1})  \\
& \quad \quad \quad \quad \cdot P_{\lambda}(1, (-q)^{-1}, (-q)^{-2}, \ldots ; q^{-1}) u^{|\lambda|}.
\end{align*}
From \cite[III.5, Example 3]{Mac}, we have the identity
$$ \prod_{i < j} \frac{1 - tx_i x_j}{1 - x_i x_j} = \sum_{\nu \atop{ \nu' \text{ even}}} c_{\nu}(t) P_{\nu}(x ; t),$$
where $c_{\nu}(t) = \prod_{i \geq 1} (1 - t)(1-t^3) \cdots (1-t^{m_i(\nu) - 1})$.  We apply this identity with the substitutions $x_i = \sqrt{-1} u (-q)^{-i}$, $t = -1$.  When $\nu'$ is even, then $m_i(\nu)$ is even for every $i$, and then $c_{\nu}(-1) = \prod_i 2^{m_i(\nu)/2} = 2^{\ell(\nu)/2}$.  We then have
\begin{align*}
\prod_{i < j} & \frac{1 - u^2/(-q)^{i+j}}{1 + u^2/(-q)^{i+j}}  = \sum_{\nu \atop{ \nu' \text{ even}}} 2^{\ell(\nu)/2} P_{\nu} (\sqrt{-1} u (-q)^{-1}, \sqrt{-1} u (-q)^{-2}, \ldots;-1) \\
 & = \sum_{\nu \atop {\nu' \text{ even}}} 2^{\ell(\nu)/2} (-1)^{|\nu|/2} q^{-|\nu|}u^{|\nu|} P_{\nu}(1, (-q)^{-1}, (-q)^{-2}, \ldots;-1).
\end{align*}
Substituting the two expansions above back into \eqref{1stExp} and finding the coefficient of $u^n$ gives the first expression for the sum of the real character degrees.

On the other hand, we can also write
\begin{align}
\prod_i & \frac{(1 + u/(-q)^i)^2}{1 + u^2/(-q)^{2i-1}} \prod_{i < j} \frac{(1 - u^2/(-q)^{i+j})^2}{(1 + u^2/(-q)^{i+j})(1 + u^2/(-q)^{i+j-1})}  \notag\\
& = \left(\prod_{i} \frac{1 + u^2/(-q)^{2i}}{1+u^2/(-q)^{2i-1}} \prod_{i < j} \frac{1 - u^2/(-q)^{i+j}}{1 + u^2/(-q)^{i+j-1}}\right) \label{1stpair} \\
& \quad \quad \cdot \left( \prod_i \frac{(1 + u/(-q)^i)^2}{1 + u^2/(-q)^{2i}} \prod_{i < j} \frac{1 - u^2/(-q)^{i+j}}{1 + u^2/(-q)^{i+j}} \right) \label{2ndpair}
\end{align}
For the product in \eqref{1stpair}, apply Theorem \ref{WarId} with
$$a = (-q)^{-1/2}, \quad b = -(-q)^{-1/2}, \quad x_i = -u q^{1/2} (-q)^{-i}, \quad t = q^{-1},$$ and we obtain
\begin{align*}
\prod_{i} & \frac{1 + u^2/(-q)^{2i}}{1+u^2/(-q)^{2i-1}} \prod_{i < j} \frac{1 - u^2/(-q)^{i+j}}{1 + u^2/(-q)^{i+j-1}} \\
& = \sum_{\lambda} (-1)^{\ell(\lambda_{\text{o}})/2} q^{-(\ell(\lambda_{\text{o}}) + |\lambda|)/2} h_{\lambda_{\text{e}}} (q^{-1} ; q^{-1}) h_{\lambda_{\text{o}}}(-1 ; q^{-1}) \\
& \quad \quad \quad \quad \cdot P_{\lambda}(1, (-q)^{-1}, (-q)^{-2}, \ldots ; q^{-1}) u^{|\lambda|}.
\end{align*}
From \cite[Equation (3.3.8)]{A}, we have
$$H_m(-1 ; q^{-1}) = \left\{ \begin{array}{ll} (q^{-1} ; q^{-2})_{m/2} & \text{ if $m$ is even}, \\ 0 & \text{ if $m$ is odd.} \end{array} \right.$$
So, the only $\lambda$ which will appear in the expansion of \eqref{1stpair} above are those such that odd parts have even multiplicity, so $(\lambda_{\text{o}})'$ is even.  Thus the coefficient of $u^m$ in the expansion of \eqref{1stpair} is
\begin{align*}
\sum_{|\lambda| = m \atop{ (\lambda_{\text{o}})' \text{ even}}} & (-1)^{\ell(\lambda_{\text{o}})/2} q^{-(\ell(\lambda_{\text{o}}) + |\lambda|)/2} h_{\lambda_{\text{e}}} (q^{-1} ; q^{-1}) \left( \prod_i (q^{-1} ; q^{-2})_{m_i(\lambda_{\text{o}})/2} \right) \\
& \cdot P_{\lambda}(1, (-q)^{-1}, (-q)^{-2}, \ldots ; q^{-1}).
\end{align*}
For \eqref{2ndpair}, we apply Theorem \ref{WarId} with $a = b = -\sqrt{-1}$, $x_i = u \sqrt{-1}(-q)^{-i}$, $t = -1$, to obtain
\begin{align*}
\prod_i & \frac{(1 + u/(-q)^i)^2}{1 + u^2/(-q)^{2i}} \prod_{i < j} \frac{1 - u^2/(-q)^{i+j}}{1 + u^2/(-q)^{i+j}} \\
& = \sum_{\nu} (-1)^{3 \ell(\nu_{\text{o}})/2} (-1)^{|\nu|/2} (-q)^{-|\nu|} h_{\nu_{\text{e}}}( -1 ; -1) h_{\nu_{\text{o}}} (1 ; -1) \\
& \quad \quad \quad \quad \cdot P_{\nu} (1, (-q)^{-1}, (-q)^{-2}, \ldots;-1) u^{|\nu|}.
\end{align*}
One may compute directly from the recursion for $H_m(z ; t)$ (see \cite[Chapter 3, Example 6]{A}) that we have
$$ H_m(z ; -1) = \left\{ \begin{array}{ll} (z^2 + 1)^{m/2} & \text{ if $m$ is even,} \\ (z+1)(z^2 + 1)^{(m-1)/2} & \text{ if $m$ is odd.} \end{array} \right.$$
So, $H_m(1 ; -1) = 2^{\lceil m/2 \rceil}$, and when $m$ is even, $H_m(-1 ; -1) = H_m(1 ; -1)$.  Since $H_m(-1 ; -1) = 0$ when $m$ is odd, then the only $\nu$ we need consider are those such that all even parts have even multiplicity, that is, $(\nu_{\text{e}})'$ is even.  It follows now that the coefficient of $u^k$ in the expansion of \eqref{2ndpair} is given by
$$ \sum_{|\nu| = k \atop{ (\nu_{\text{e}})' \text{ even}}}  (-1)^{(\ell(\nu_{\text{o}}) + |\nu|)/2} q^{-|\nu|} \left( \prod_i 2^{\lceil m_i(\nu)/2 \rceil} \right) P_{\nu}(1, (-q)^{-1}, (-q)^{-2}, \ldots;-1).$$
Using the expansions for \eqref{1stpair} and \eqref{2ndpair} we have obtained gives the second desired expression for the real character degree sum.
\end{proof}

We note that we could now give results for $q$ odd which parallel Corollaries \ref{UnSumEven+-} and \ref{GenFnEvenAlt}, but we omit them here.  We conclude with an example of applying Theorem \ref{UnSumOdd} to $U(2,q)$ with $q$ odd.  We use the second expression in Theorem \ref{UnSumOdd}, which is a sum with three terms, corresponding to $\lambda = (2)$ with  $\nu = (0)$, $\lambda = (1^2)$ with  $\nu=(0)$, and $\lambda = (0)$ with $\nu = (1^2)$.  We can use the previously calculated values of Hall-Littlewood functions (recall that $P_{(1^m)}(x ; t)$ is independent of $t$), and together with the facts that $h_{(2)}(q^{-1} ; q^{-1}) = (q+1)/q$, and $(q^{-1} ; q^{-2})_1 = (q-1)/q$, we find that the sum of the degrees of the real-valued characters of $U(2,q)$, $q$ odd, is:
$$ (q^2 - 1)(q+1) \left( \frac{q^2+1}{q(q^2 - 1)} + \frac{1}{q(q+1)^2} + \frac{-2}{(q+1)(q^2 - 1)} \right) = q^2 +q.$$
From Proposition \ref{involUodd}, the sum of the degrees of characters with Frobenius-Schur indicator $1$, minus the degree sum of those with Frobenius-Schur indicator $-1$, is $q^2 - q + 2$.  This gives that the sum of the degrees of characters with Frobenius-Schur indicator $-1$ is $q-1$.  From the character degrees of $U(2,q)$, this is the minimal possible degree greater than $1$, which means there is a unique character of degree $q-1$ with Frobenius-Schur indicator $-1$.  This is known from the character table of $U(2,q)$, as mentioned at the end of the paper of Gow \cite{Gow}.  More generally, it is known \cite{SV} that $U(2m,q)$, $q$ odd, has $q^{m-1}$ semisimple characters with Frobenius-Schur indicator $-1$, which is exactly this character when $m=1$.

\section*{Acknowledgements} The authors thank Ole Warnaar for pointing out how his result from \cite{War} may be applied to expand the generating function in our Theorem \ref{degreesU}.  Fulman was supported by a Simons Foundation Fellowship and NSA grant H98230-13-1-0219, and Vinroot was supported by NSF grant DMS-0854849.


\begin{thebibliography}{99}

\bibitem {A} Andrews, G., The theory of partitions. Cambridge University Press, Cambridge, 1984.

\bibitem{BT} Bannai, E. and Tanaka, H., The decomposition of the permutation character $1_{GL(n,q^2)}^{GL(2n,q)}$, {\it J. Algebra} {\bf 265} (2003), no. 2, 496-512.

\bibitem{Ca} Carter, R. W., Finite groups of Lie type. Conjugacy classes and complex characters, Pure and Applied Mathematics (New York), John Wiley \& Sons, Inc., New York, 1985.

\bibitem{Ch} Chigira, N., The solutions of $x^d = 1$ in finite groups, {\it J. Algebra} {\bf 180} (1996), no. 3, 653-661.

\bibitem{CHM} Chowla, S., Herstein, I. N., and Moore, W. K., On recursions connected with symmetric groups. I, {\it Canad. J. Math.} {\bf 3} (1951), 328-334.

\bibitem {FNP} Fulman, J., Neumann, P., and Praeger, C., A generating function approach to the enumeration of matrices
in classical groups over finite fields, {\it Mem. Amer. Math. Soc.} {\bf 176} (2005), vi+90 pp.

\bibitem{Gow} Gow, R., On the Schur indices of characters of finite classical groups, {\it J. London Math. Soc. (2)} {\bf 24} (1981), no. 1, 135-147.

\bibitem {GV} Gow, R. and Vinroot, C. R., Extending real-valued characters of finite general linear and unitary groups
on elements related to regular unipotents, {\it J. Group Theory} {\bf 11} (2008), 299-331.

\bibitem{Isa} Isaacs, I. M., Character theory of finite groups. Corrected reprint of the 1976 original.  AMS Chelsea Publishing, Providence, RI, 2006.

\bibitem {Mac} Macdonald, I. G., Symmetric functions and Hall polynomials. Second edition. Clarendon Press, Oxford, 1995.

\bibitem {Mo} Morrison, K., Integer sequences and matrices over finite fields, {\it J. Integer Sequences} {\bf 9}
 (2006), Article 06.2.1., 28 pp.

\bibitem{Ohm77} Ohmori, Z., On the Schur indices of $\mathrm{GL}(n, q)$ and $\mathrm{SL}(2n+1, q)$, {\it J. Math. Soc. Japan} {\bf 29} (1977), no. 4, 693-707.

\bibitem{Ohm81} Ohmori, Z., On the Schur indices of reductive groups, II, {\it Quart. J. Math. Oxford Ser. (2)} {\bf 32} (1981), no. 128, 443-452.

\bibitem{Ohm} Ohmori, Z., The Schur indices of the cuspidal unipotent characters of the finite unitary groups, {\it Proc. Japan Acad. Ser. A Math. Sci.} {\bf 72} (1996), 111-113.

\bibitem {Sp} Springer, T., Regular elements of finite reflection groups, {\it Invent. Math.} {\bf 25} (1974), 159-198.

\bibitem {SV} Srinivasan, B. and Vinroot, C. R., Semisimple symplectic characters of finite unitary groups, {\it J. Algebra} {\bf 351} (2012), no. 1, 459-466.

\bibitem {TV} Thiem, N. and Vinroot, C. R., On the characteristic map of the finite unitary groups, {\it Adv. Math.}
{\bf 210} (2007), no. 2, 707-732.

\bibitem {Wall} Wall, G. E., On the conjugacy classes in the unitary, symplectic and orthogonal groups,
{\it J. Austral. Math. Soc.} {\bf 3} (1963), 1-62.

\bibitem{War} Warnaar, S. O., Rogers-Szeg\H{o} polynomials and Hall-Littlewood symmetric functions, {\it J. Algebra} {\bf 303} (2006), no. 2, 310-330.

\bibitem{Zel} Zelevinsky, A. V., Representations of finite classical groups. A Hopf algebra approach.  Lecture Notes in Mathematics 869, Springer-Verlag, Berlin-New York, 1981.

\end{thebibliography}
\end{document}